\documentclass[ECP,preprint]{ejpecp}

\usepackage{graphicx} 
\usepackage{amsmath, relsize}
\usepackage{amsthm}
\usepackage{amsfonts}
\usepackage{amssymb}
\usepackage{array}
\usepackage{url}
\usepackage{dsfont}
\usepackage{mathtools}
\usepackage{xcolor}
\DeclareRobustCommand{\rev}[1]{{#1}}
\usepackage{tikz}
\usepackage{enumerate}
\usepackage{setspace}
\usepackage[scr=boondox,  
cal=esstix]   
{mathalpha}

\SHORTTITLE{Extremes of the GFF on regular graphs}

\TITLE{Extremes of the zero-average Gaussian Free Field on random regular graphs
    \thanks{The
research of L.H. was supported by the Deutsche Forschungsgemeinschaft (DFG, German Research
Foundation) within the framework of the collaborative research center "Multiscale Simulation Methods for Soft-Matter Systems" (TRR 146) under Project No. 233630050 and Project-ID 443891315 within SPP 2265. The research of C.M.\ was partly funded by the German Research Foundation (project number 443916008, priority program SPP 2265). A.K.'s research is funded by the Cusanuswerk (Bischöfliche Studienförderung).}
    } 

\AUTHORS{%
Lisa~Hartung \footnote{Johannes Gutenberg University Mainz, Germany.
    \EMAIL{lhartung@uni-mainz.de}}\orcid{0000-0002-3420-551X}
\and
      Andreas Klippel \footnote{Johannes Gutenberg University Mainz, Germany.
          \EMAIL{anklippe@uni-mainz.de}}
\and 
  Christian M\"onch \footnote{Independent Researcher, Germany.
      \EMAIL{cmoench25@gmail.com}}\orcid{0000-0002-6531-6482}
}


\AMSSUBJ{60G70} 
\AMSSUBJSECONDARY{60K35; 05C80} 

\SUBMITTED{2025-11-18} 
\ACCEPTED{??????} 

\hypersetup{pdftitle={Extremes of the zero-average Gaussian Free Field on random regular graphs},
  pdfauthor={Lisa Hartung, Andreas Klippel, Christian M{\"o}nch}}

\VOLUME{0}
\YEAR{????}
\PAPERNUM{0}
\DOI{10.1214/YY-TN}

\ABSTRACT{We study the extreme value statistics of the zero-average Gaussian free field (GFF) on random $r$-regular graphs and the Gaussian free field on $r$-regular trees. For random $r$-regular graphs of diverging size, for every fixed $r\ge3$, we show that the rescaled extremal point process of the field is asymptotically distributed, in the annealed sense, as a Poisson point process on the line with intensity $e^{-x}\,\mathrm{d}x$. The same limit behaviour is obeyed by the restriction of the GFF on $r$-regular trees to finite subsets of vertices. Our approach relies on a direct Gaussian comparison argument and precise Green function estimates.
    }

\newtheorem{prop}[theorem]{Proposition}

\renewcommand{\P}{\mathbb{P}}
\newcommand{\x}{\mathbf{x}}

\newcommand{\G}{\mathscr{G}}

\newcommand{\E}{\mathbb{E}}

\newcommand{\R}{\mathbb{R}}

\newcommand{\T}{\mathscr{T}}

\newcommand{\M}{\mathscr{M}}

\newcommand{\N}{\mathbb{N}}
\providecommand{\orcid}[1]{}

\newcommand{\zigzag}[1]{%
    \begin{tikzpicture}[#1]%
        \draw (0,0) -- (0.5ex,0ex);%
        \draw (0.5ex,0ex) -- (0.5ex,1ex);%
        \draw (0.5ex,1ex) -- (1ex,1ex);%
    \end{tikzpicture}%
}

\newcommand{\Clbrc}{C^\infty_{\mbox{\zigzag{}}}}

\numberwithin{equation}{section}

\begin{document}
\raggedbottom

\section{Introduction and main result}
In recent years, properties of  \rev{level-set }percolation and local properties of the zero-average Gaussian free field (GFF) on finite graphs and trees have been studied intensively; see, for example,  \rev{\cite{AbaechCerny_trees,AbacherliCerny2,CernyIII,DrewitzGalloPrev,sznitman2015disconnection,sznitman2016coupling}. }In the present article, we focus on the  \rev{extreme-value }statistics of such fields on random $r$-regular graphs and identify the proper scaling of the extremes.

\rev{Related results for the GFF on $\mathbb Z^d$, $d\ge3$, and for the zero-average GFF on the discrete torus were obtained in \cite{chiarini2015,chiarini2016,DasHazra2019}. Similar Poisson and Gumbel limits were established for random interlacements \cite{Belius12} and random walk on discrete tori \cite{Belius13}. The maximum of the membrane model on regular trees was studied in \cite{CiprianiDanHazraRay23}, where Gumbel fluctuations for the maximum were proved.
}

The only previous results in this direction for regular graphs are the rough upper bound in \cite{Conchon23} and the work \cite{DasHazra2019} in the rigid setting of the Euclidean torus. \rev{Whereas }\cite{DasHazra2019} relies on a coupling, we develop a different approach: we combine a general Gaussian comparison principle for point processes (Theorem~\ref{thm:gauss-comp-ppp}) with estimates for the Green function on random $r$-regular graphs (Section~4.1) that complement those in \cite{AbacherliCerny2}. \rev{In the proof of the comparison principle, we use Lemma~\ref{lem:gauss-interp}, cf.\ \cite[Lemma~4.1]{bovier2005extreme}, and Lemma~\ref{lem:pp-equivalence}, taken from \cite[Lemma~4.1]{berestycki2018extremes}. }We believe both \rev{the comparison principle and the Green-function estimates }to be of independent interest.

Throughout, we let $(\G_N,o_N)$ denote a uniformly rooted $r$-regular graph on $N\in 2\N$ vertices. The corresponding distribution on finite graphs is denoted by $\mathbf{P}$ and the graph distances in $\G_N$ are written as $d_{\G_N}(\cdot,\cdot)$. We work under the law
\[
  \P_N \;=\; \int \P^{\G_N}\, \mathrm{d}\mathbf P(\G_N),
\]
where $\P^{\G_N}$ is the process law conditional on the graph realisation $\G_N$. Given $\G_N$, we consider the canonical simple random walk $(X_k)_{k\ge0}$,
its continuous-time counterpart $(\overline X_t)_{t\ge0}$, \rev{obtained by placing independent mean-one exponential clocks on the vertices,
}and the corresponding zero-average GFF $\Psi_N$ on $\G_N$.
The latter is defined as the centred Gaussian field with covariances given by
\[
  G_{\G_N}(x, y)
  := \int_0^\infty \big( \P_x^{\G_N}(\overline{X}_t = y) - \tfrac{1}{N} \big)\, \mathrm{d}t,
  \qquad x, y \in \G_N,
\]
that is, the \emph{zero-average Green function} of $\G_N$. \rev{Since $\G_N$ is $r$-regular, the stationary distribution of the walk is uniform. Thus, the term $1/N$ in the definition of $G_{\G_N}$ subtracts the stationary distribution. In particular,
}\[
\rev{\sum_{y\in V_N}G_{\G_N}(x,y)=0,\qquad x\in V_N,
}\]
\rev{and hence $\sum_{x\in V_N}\Psi_N(x)=0$ almost surely, which is the zero-average condition.
}

\rev{This is different from imposing a Dirichlet boundary condition at the root of a finite rooted tree, while keeping the boundary elsewhere free. The corresponding thermodynamic limit is a Gaussian branching random walk with independent centred Gaussian increments along the edges, whose variance grows linearly with the distance from the root. By contrast, the infinite-volume GFF on the $r$-regular tree considered below has no boundary condition at the root and constant variance. For the extremal behaviour of branching random walks, see \cite{Aidekon13,Madaule17}.}
We further set
\begin{equation}\label{eq:rescaling}
  \sigma_r^2 := \frac{r-1}{r-2},
  \qquad
  a_N := \sqrt{2\log N} - \frac{\log\log N + \log(4\pi)}{2\sqrt{2\log N}},
  \qquad
  b_N := a_N^{-1},
\end{equation}
and define the rescaled field $Z_x := \frac{\Psi_N(x)}{\sigma_r}$, for  $x\in V_N$. The associated extremal point process and centred maximum are given by
\[
  \mathcal P_N := \sum_{x\in V_N}\delta_{\frac{Z_x-a_N}{b_N}},
  \qquad
  M_N := \max_{x\in V_N}\frac{Z_x-a_N}{b_N}.
\]
Let further $\mathcal P_\infty$ be a Poisson point process on $\R$ with intensity $e^{-x}\,\mathrm{d}x$,
and set
\[
  M_\infty := \max \mathcal P_\infty.
\]
\rev{Point-process convergence is understood with respect to the vague topology. If a point process $P$ is bounded from above, $\max P$ denotes its largest atom. }Denoting by $\Rightarrow$ convergence in distribution, our main result can now be stated as follows:
\begin{theorem}[Extremal point process and maximum on the random $r$-regular graph]\label{thm:graph-ext-proc}
Let $r\ge3$. Then, with the notation above, as $N\to\infty$, there exist sets of graphs
$\Omega_N,\,N\in\N$, such that $\mathbf{P}(\Omega_N^\mathsf{c})=o(1)$ and on $\Omega_N$
\[
  (\mathcal P_N,\,M_N)\ \overset{\P^{\G_N}}\Longrightarrow\ (\mathcal P_\infty,\,M_\infty).
\]
Consequently, $(\mathcal P_N,\,M_N)\ \overset{\P_N}\Longrightarrow\ (\mathcal P_\infty,\,M_\infty)$,
which in particular implies that $M_N$ is asymptotically Gumbel-distributed, so that
\[
  \P_N\Big( \max_{x\in \G_N} \Psi_N(x)\ \le\
  \sigma_r\Big(a_N+t b_N\Big)\Big)
  \ \longrightarrow\ \exp(-e^{-t}).
\]
\end{theorem}

Our proof of Theorem~\ref{thm:graph-ext-proc} relies on two key ingredients.
\begin{itemize}
    \item[(1)] A Gaussian comparison principle for the joint convergence of the rescaled extremal process and the maximum, stated and proven in Section 2. In Section 3 we showcase the usefulness of this method by deploying it for the rescaled extremal process and the rescaled maximum of the Gaussian free field on finite subsets of the $r$-regular infinite tree, see Theorem \ref{thm:tree-ext-proc}.
    \item[(2)] Precise `almost uniform' Green function: In $\mathbf{P}$-probability, the Green function on $\G_N$ can locally be well approximated by the one on the  \rev{rooted infinite $r$-regular }tree for asymptotically almost all vertices. See Section 4.1 for the details.
\end{itemize}
Finally, Section 4.2 contains the proof of Theorem \ref{thm:graph-ext-proc}.

\section{A Gaussian comparison principle}
Our first result is a Gaussian comparison principle for the joint convergence of the extremal process and the maximum, which is instrumental in deriving our main results. Gaussian comparison has proved powerful in the study of extremes of Gaussian fields on graphs, cf.\ \cite{berestycki2018extremes,bramson2012tightness,bramson2016convergence,chiarini2015,chiarini2016,DingRoyZeit17,BiskupLouidor18}.

Throughout this section,  \rev{for $q\in\{0,1\}$ we let
}\[
\rev{X(q)=\big(X_{N,i}(q):N\in\N,\ 1\le i\le N\big)
}\]
denote  Gaussian triangular arrays \rev{, and assume that $X(0)$ and $X(1)$ are independent}. Their interpolation is given  \rev{by
}\[
\rev{X(h)=\sqrt{1-h}\,X(0)+\sqrt h\,X(1),\qquad h\in[0,1],
}\]
where  \rev{the operations are }understood entry-wise. For  $N\in\N$,  \rev{write
}\[
\rev{X_N(h)=\big(X_{N,1}(h),\dots,X_{N,N}(h)\big)
}\]
for the $N$-th row  \rev{and assume that it }is a centred Gaussian vector in $\R^N$ \rev{. Its covariance matrix is
}\[
\rev{\Sigma(N,h)=\operatorname{Cov}\big(X_N(h)\big)
=(1-h)\operatorname{Cov}\big(X_N(0)\big)+h\operatorname{Cov}\big(X_N(1)\big),\qquad h\in[0,1].
}\]

\begin{theorem}\label{thm:gauss-comp-ppp}
    Assume that, as $N\to\infty,$
    \begin{itemize}
        \item \rev{the diagonal variances agree and are positive, that is,
        }\[
        \rev{\Sigma_{ii}(N,0)=\Sigma_{ii}(N,1)>0,\qquad 1\le i\le N;
        }\]
        \item there are sequences $(a_N)_{N\in\N}$ and $(b_N)_{N\in\N}$ \rev{, with $b_N>0$, such that
        }\begin{equation}\rev{\label{eq.ppp}
            \left(\sum_{i=1}^{N} \delta_{(X_{N,i}(0)-a_N)/b_N}, \max_{1\le i\le N} \frac{X_{N,i}(0)-a_N}{b_N}\right)
            \;\Longrightarrow\; \left(\mathcal{P}_\infty, \max \mathcal{P}_\infty\right),
        }\end{equation}
        where $\mathcal{P}_\infty$ is a point process on $\R$ with $\mathcal{P}_\infty\big((0,\infty)\big)<\infty$;
        \item for every bounded and open set $S\subset\R$,
         \begin{equation}\rev{\label{eq:tozero}
            \begin{aligned}
            &\frac{1}{b_N^{2}}\int_0^1\!\sum_{\substack{1\le i\ne j\le N}}
            \big|\Sigma_{ij}(N,1)-\Sigma_{ij}(N,0)\big|\\
            &\qquad\times
            \P\left(\tfrac{X_{N,i}(h)-a_N}{b_N}\in S,\
            \tfrac{X_{N,j}(h)-a_N}{b_N}\in S\right)\,\mathrm{d}h
            \ \xrightarrow[N\to\infty]{}\ 0.
            \end{aligned}
        }\end{equation}
    \end{itemize}
    Then, as $N\to\infty,$
      \begin{equation}
        \rev{\left(\sum_{i=1}^{N} \delta_{(X_{N,i}(1)-a_N)/b_N}, \max_{1\le i\le N} \frac{X_{N,i}(1)-a_N}{b_N}\right)
        \;\Longrightarrow\; \left(\mathcal{P}_\infty, \max \mathcal{P}_\infty\right).
    }\end{equation}
\end{theorem}
Before addressing the proof, we collect a few auxiliary results. We start with a consequence of the Gaussian integration-by-parts formula, see e.g.\ \cite[Lemma 4.1]{bovier2005extreme} for a detailed derivation.
\begin{lemma}\label{lem:gauss-interp}
    Let $Y(0)=\big(Y_1(0),\dots,Y_n(0)\big)$ and $Y(1)=\big(Y_1(1),\dots,Y_n(1)\big)$ be independent centred Gaussian vectors in $\mathbb{R}^n$ with covariance matrices  \rev{$\Lambda^{(q)}=\big(\Lambda_{ij}^{(q)}\big)$}, \rev{$q\in\{0,1\}$, satisfying $\Lambda_{ii}^{(q)}=1$ for all $q\in\{0,1\}$ and $1\le i\le n$. Set, }for $h\in(0,1)$\rev{,
    }\[
    Y(h)= \sqrt{h} Y(1)+\sqrt{1-h}\,Y(0),
    \qquad
    \Lambda(h)= \operatorname{Cov}\big(Y(h)\big),
    \]
    and let $F\in C^{2}(\R^n,\R)$ be any function satisfying the moderate growth condition
    \[
    |F(x)|+\|\nabla F(x)\|+\|\!\operatorname{Hess}F(x)\|\le C\,\big(1+\|x\|\big)^{m},\quad x\in\R^n,
    \]
    for some $C<\infty, m\in\N$, where $\|\cdot\|$ denotes, for definiteness, the $\max$-norm.
    Then the map $h\mapsto \E\big[F\big(Y(h)\big)\big]$ is continuously differentiable on $[0,1]$ and
     \begin{equation}\rev{\label{eq:interp-integral}
        \E\big[F\big(Y(1)\big)-F\big(Y(0)\big)\big]
        \;=\;\frac{1}{2}\sum_{i,j=1}^{n}\int_{0}^{1}\bigl(\Lambda_{ij}^{(1)}-\Lambda_{ij}^{(0)}\bigr)\,
        \E\big[\partial_{ij}F\big(Y(h)\big)\big]\,\mathrm{d}h,
    }\end{equation}
    where $\partial_{ij}F, 1\leq i,j\leq n$, denote the entries of $\operatorname{Hess} F$.
\end{lemma}

We say that a function $\varphi:\R\to\R$ has \emph{l-bounded support} if there exists $K<\infty$, such that $\varphi(x)=0$ for all $x\leq-K$. A differentiable function $\varphi$ is \emph{eventually constant (on the right)} if $x\mapsto\varphi'(-x)$ has l-bounded support. We denote by $\Clbrc$ the cone of all non-decreasing, eventually constant, smooth functions with l-bounded support. Note that $\varphi\in\Clbrc$ implies that $\operatorname{supp}(\varphi')$ is a bounded open set and, conversely, for every bounded open set $S$ we can find a function $\varphi\in \Clbrc$ such that $\operatorname{supp}(\varphi')=S$. 

\begin{lemma}[{\cite[Lemma 4.1, parts (i),(ii), and (iv)]{berestycki2018extremes}}]\label{lem:pp-equivalence}

    Let $(P_t: t\in \N\cup\{\infty\})$ denote a family of point processes on $\mathbb{R}$ where $P_\infty\big((0,\infty)\big) < \infty$ almost surely.
    The following statements are equivalent:
    \begin{enumerate}[(i)]
        \item $(P_t, \max P_t) \Longrightarrow(P_\infty, \max P_\infty)$ jointly, as $t\to\infty$;
        \item $P_t \Longrightarrow P_\infty$ and $\max P_t \Longrightarrow \max P_\infty$, as $t\to\infty$;
        \item  \rev{$\mathbb{E}\left[e^{-\langle P_t,\varphi \rangle}\right] \;\to\;
        \mathbb{E}\left[e^{-\langle P_\infty,\varphi\rangle}\right]$, }as $t\to\infty$,
        for all $\varphi\in\Clbrc$.
    \end{enumerate}
\end{lemma}

\begin{proof}[Proof of Theorem~\ref{thm:gauss-comp-ppp}]
\rev{For $h\in[0,1]$, set
}\[
\rev{\mathcal P_N(h):=\sum_{i=1}^N\delta_{(X_{N,i}(h)-a_N)/b_N}.
}\]
By Lemma~\ref{lem:pp-equivalence} and~\eqref{eq.ppp}, it suffices to prove that
\[
\Big|
\E e^{-\langle \mathcal{P}_{N}(1),\varphi\rangle}
-
\E e^{-\langle \mathcal{P}_{N}(0),\varphi\rangle}
\Big|
\longrightarrow 0
\quad\text{for all }\varphi\in\Clbrc.
\]
Fix $\varphi\in\Clbrc$ and set $S=\operatorname{supp}(\varphi')$. For $x=(x_1,\dots,x_N)\in\R^N$, define
\[
F_N(x):=\exp\Big(-\sum_{k=1}^N \varphi\Big(\frac{x_k-a_N}{b_N}\Big)\Big),
\quad\text{so that}\quad
e^{-\langle \mathcal{P}_N(h),\varphi\rangle}=F_N(X_N(h)).
\]

Let  \rev{$\sigma_{N,i}^2=\operatorname{Var}(X_{N,i}(0))=\operatorname{Var}(X_{N,i}(1))$ and put
$Z_{N,i}(h)=X_{N,i}(h)/\sigma_{N,i}$}. Then $Z_N(0)$ and $Z_N(1)$ are independent centred Gaussian vectors with unit  \rev{variances, hence so is every
$Z_N(h)=\sqrt{1-h}\,Z_N(0)+\sqrt h\,Z_N(1)$. Defining
}\[
\rev{\widetilde F_N(z)
:=F_N(\sigma_{N,1}z_1,\dots,\sigma_{N,N}z_N)
=\exp\Big(-\sum_{k=1}^N
\varphi\Big(\frac{\sigma_{N,k}z_k-a_N}{b_N}\Big)\Big),
}\]
we have $F_N(X_N(h))=\widetilde F_N(Z_N(h))$. Applying Lemma~\ref{lem:gauss-interp}  \rev{and using the equality of the diagonal variances gives
}\[
\rev{\E[\widetilde F_N(Z_N(1))-\widetilde F_N(Z_N(0))]
=\frac{1}{2}\sum_{i\ne j}(\widetilde\Lambda_{ij}(N,1)-\widetilde\Lambda_{ij}(N,0))
\int_0^1\E[\partial_{ij}\widetilde F_N(Z_N(h))]\,\mathrm{d}h,
}\]
 \rev{where $\widetilde\Lambda_{ij}(N,h)=\Sigma_{ij}(N,h)/(\sigma_{N,i}\sigma_{N,j})$. For }$i\ne j$ \rev{,
}\[
\rev{\partial_{ij}\widetilde F_N(z)
=\frac{\sigma_{N,i}\sigma_{N,j}}{b_N^2}\,
\varphi'\Big(\frac{\sigma_{N,i}z_i-a_N}{b_N}\Big)
\varphi'\Big(\frac{\sigma_{N,j}z_j-a_N}{b_N}\Big)
\widetilde F_N(z).
}\]
 \rev{The factors $\sigma_{N,i}\sigma_{N,j}$ cancel. Since $0\le\widetilde F_N\le1$ and $\varphi'$ vanishes off $S$, we obtain
}\begin{align*}
&\rev{\Big|\E[F_N(X_N(1))-F_N(X_N(0))]\Big|}\\
&\rev{\quad\le\frac{\|\varphi'\|_\infty^2}{2b_N^2}
\int_0^1\sum_{i\ne j}
|\Sigma_{ij}(N,1)-\Sigma_{ij}(N,0)|
\P\Big(\tfrac{X_{N,i}(h)-a_N}{b_N}\in S,\tfrac{X_{N,j}(h)-a_N}{b_N}\in S\Big)\,\mathrm{d}h.
}\end{align*}
By~\eqref{eq:tozero}, the right-hand side tends to $0$ \rev{, which proves }the claim.
\end{proof}

\begin{remark}
\rev{The double sum in \eqref{eq:tozero} is closely related to a second-moment computation. Indeed, for fixed $h\in[0,1]$, the off-diagonal contribution to the second moment of the number of points in the extremal window $S$ is}
\[
\rev{
\sum_{1\le i\ne j\le N}
\P\left(
\frac{X_{N,i}(h)-a_N}{b_N}\in S,\,
\frac{X_{N,j}(h)-a_N}{b_N}\in S
\right).
}
\]
\rev{A first-moment calculation identifies the natural centring and gives an upper bound for the maximum, while a second-moment calculation involves essentially the same pairwise estimates as \eqref{eq:tozero} and can provide a matching lower bound. Such moment arguments locate the maximum but do not by themselves yield convergence of its law or of the full extremal point process. The advantage of the Gaussian comparison method is that, once these estimates show that the quantity in \eqref{eq:tozero} tends to zero, Theorem~\ref{thm:gauss-comp-ppp} directly yields convergence of the full extremal point process and hence the Poisson point-process limit in our applications.}
\end{remark}

The following simple analytical lemma provides an upper bound for the probability in \eqref{eq:tozero}, which we use to apply Theorem~\ref{thm:gauss-comp-ppp}.
\begin{lemma}\label{lem:doublegauss}
    Let $(X,Y)$ be a centered bivariate Gaussian vector with unit variances and correlation  \rev{$\rho \in (-1,1)$}.
    Let $S \subset \mathbb{R}$ be a bounded open set with Lebesgue measure $|S| < \infty$. 
    For $u > 0$, define
    \[
    d(u,S) := \inf_{s \in S} |u + s| = \operatorname{dist}(-u, S).
    \]
    Then
 \[
    \rev{\mathbb{P}\big(X - u \in S,\; Y - u \in S\big)
        \;\le\;
        \frac{|S|^{2}}{2\pi\sqrt{1-\rho^{2}}}\,
        \exp\!\Big(-\frac{d(u,S)^{2}}{1+|\rho|}\Big).
    }\]
\end{lemma}

\begin{proof}
We have
    \[
    \mathbb{P}(X-u\in S,\ Y-u\in S)
    = \iint_{(u+S)\times(u+S)} f(x,y)\,dx\,dy,
    \]
    where
    \[
    f(x,y)
    = \frac{1}{2\pi\sqrt{1-\rho^{2}}}
    \exp\!\left(-\frac{x^{2}-2\rho xy+y^{2}}{2(1-\rho^{2})}\right),
    \]
         \rev{is }the joint density of $X$ and $Y$. For  \rev{$\rho \in (-1,1)$}, we have the elementary inequality
         \[
        \rev{x^{2}-2\rho xy+y^{2} \ge (1-|\rho|)(x^{2}+y^{2}),
        }\]
         \rev{which follows from $2|xy|\le x^2+y^2$}.
    For every $x \in u+S$, we have $|x| \ge d:=d(u,S)$, and the same holds for $y \in u+S$. Therefore,
 \begin{equation}\rev{\label{eq:inequality}
    \inf_{(x,y)\in (u+S)\times(u+S)}
        \frac{x^{2}-2\rho xy+y^{2}}{2(1-\rho^{2})}
        \ge
        \frac{d^{2}+d^{2}}{2(1+|\rho|)}
        = \frac{d^{2}}{1+|\rho|}.
}\end{equation}
        Using \eqref{eq:inequality} to uniformly bound the density and  \rev{observing }that the integration region has area $|S|^{2}$, the claim follows.
\end{proof}
\section{Extrema of the Gaussian free field on the regular tree}
Our first result about the extremal behaviour of the GFF concerns the infinite $r$-regular tree $\T=\T_r$ for $r\geq 3$. Let $X=(X_k)_{k=0}^\infty$ denote simple random walk on $\T$ and let $\P_x$ and $\E_x$ denote its law and expectation given $X_0=x$. The  \rev{associated }Green function is then given by
 \begin{equation}\rev{\label{eq:tree-green}
    g(x,y) := }{\rev{\E}}\rev{_x\left[\sum_{k=0}^\infty \mathds{1}_{\{X_k = y\}}\right] = \frac{r-1}{r-2} \left(\frac{1}{r-1}\right)^{d_{\T}(x,y)} \quad \text{for } x, y \in \T,
}\end{equation}
 \rev{see, for example, \cite[(1.6)]{AbaechCerny_trees}. Here }$d_{\T}(\cdot,\cdot)$ denotes graph distance in $\T$. \rev{Since the continuous-time walk has mean-one holding times, its occupation-time Green function agrees with the discrete-time Green function above. }Let $\psi=\big(\psi(x), x\in \T\big)$ denote the GFF on $\T$, i.e.\ the centred Gaussian field with covariance
 \[
\rev{\E[\psi(x)\psi(y)]=g(x,y),\quad x,y\in \T.
}\]
Throughout, we assume that $V_N\subset\T$ is a subtree with $N$ vertices.  Recall that $\sigma_r^2$
, $a_N$ and $b_N$ are defined in \eqref{eq:rescaling}.
Define the centred and rescaled field
\[
Z_x \;:=\; \frac{\psi(x)}{\sigma_r},
\qquad x\in V_N,
\]
so that $\operatorname{Var}(Z_x)=1$.
Let
\[
\mathcal P_N \;:=\; \sum_{x\in V_N}\delta_{\frac{Z_x-a_N}{b_N}},
\qquad
M_N \;:=\; \max_{x\in V_N}\frac{Z_x-a_N}{b_N}.
\]
Recall that  $\mathcal P_\infty$ is a Poisson point process on $\R$ with intensity $e^{-x}\,\mathrm{d}x$, and
$
M_\infty:=\max \mathcal P_\infty.
$

\begin{theorem}[Extremal point process and maximum on the $r$-regular tree]\label{thm:tree-ext-proc}
    Let $r\ge3$, 
    then, as $N\to\infty$,
    \[
    (\mathcal P_N,\,M_N)\ \Longrightarrow\ (\mathcal P_\infty,\,M_\infty).
    \]
    In particular, $M_N$ is asymptotically standard-Gumbel distributed.
\end{theorem}

\begin{proof}
    We wish to apply Theorem~\ref{thm:gauss-comp-ppp}.
    Recall that
    \[
    \rho_{xy}:=\E[Z_xZ_y]
    =\left(\frac{1}{r-1}\right)^{d_\T(x,y)}\in[0,1)
    \qquad \text{ for all }x,y\in V_N \text{ with }x\neq y.
    \]
Note that $\rho_1 = 1/(r-1) < 1$ if and only if $r \ge 3$.
    Let $(X_x(0))_{x\in V_N}$ be i.i.d.\ standard Gaussian variables, independent of $\psi$, and define
$
    X_x(1):=Z_x,\; x\in V_N.
$
    For $h\in[0,1]$ set
    \[
    X_x(h):=\sqrt{1-h}\,X_x(0)+\sqrt{h}\,X_x(1),\qquad
    X_N(h):=(X_x(h))_{x\in V_N}\in\R^N .
    \]
    Let $(a_N,b_N)$ be as in \eqref{eq:rescaling}.
By classical extreme value theory, see e.g.\ \cite[Thm.\ 5.5.2]{bovier2005extreme}, we have
\[
\Bigg(\sum_{x\in V_N} \delta_{(X_x(0)-a_N)/b_N},\
\max_{x\in V_N}\frac{X_x(0)-a_N}{b_N}\Bigg)
\Longrightarrow (\mathcal P_\infty,\max \mathcal P_\infty),
\]
and consequently condition~\eqref{eq.ppp} is satisfied.
It remains to show that \eqref{eq:tozero} holds, i.e.\ that for any fixed but arbitrary bounded open set $S\subset\R$ with $|S|<\infty$
\[
T_N
:=
\frac{1}{b_N^{2}}\int_0^1\sum_{x\ne y}
\big|\Sigma_{xy}(N,1)-\Sigma_{xy}(N,0)\big|\,
\P\Big(\tfrac{X_{x}(h)-a_N}{b_N}\in S,\,
\tfrac{X_{y}(h)-a_N}{b_N}\in S\Big)\, \mathrm{d}h
\]
vanishes as $N\to\infty$. To this end, we first observe that, for $x\ne y$ and $h\in[0,1]$,
$(X_{x}(h),X_{y}(h))$ is a centred Gaussian vector with unit variances
    and correlation $\rho_{xy}(h)=h\rho_{xy}\in[0,1)$.
     \rev{Let $R_S:=\sup\{|s|:s\in S\}<\infty$. Since $(X-a_N)/b_N\in S$ is equivalent to $X-a_N\in b_NS$, we have $|b_NS|=b_N|S|$ and
    }\[
    \rev{d(a_N,b_NS)=\operatorname{dist}(-a_N,b_NS)\ge a_N-b_NR_S.
    }\]
    Lemma~\ref{lem:doublegauss}  \rev{therefore yields
        }\begin{equation}\rev{\label{eq:window-tree-correct}
            \P\Big(\tfrac{X_{x}(h)-a_N}{b_N}\in S,\,
            \tfrac{X_{y}(h)-a_N}{b_N}\in S\Big)
            \ \le\
            \frac{b_N^{2}|S|^{2}}{2\pi\sqrt{1-\rho_{xy}(h)^2}}\,
            \exp\Big(-\frac{(a_N-b_NR_S)^2}{1+\rho_{xy}(h)}\Big).
        }\end{equation}
     \rev{Because $a_Nb_N=1$, the numerator is $a_N^2-O_S(1)$. Hence, with
    }\[
    \rev{H_N(q):=\frac{1}{\sqrt{1-q^2}}\exp\Big(-\frac{a_N^2}{1+q}\Big),\qquad q\in}[\rev{0,1),
    }\]
    \rev{the right-hand side of \eqref{eq:window-tree-correct} is at most $C_Sb_N^2H_N(\rho_{xy}(h))$ for all sufficiently large $N$, where $C_S<\infty$ is independent of $x,y,h$, and $N$. Consequently,
        }\[\rev{T_N
        \le
        C_S\int_0^1\sum_{x\ne y}
        \big|\Sigma_{xy}(N,1)-\Sigma_{xy}(N,0)\big|\, \frac{1}{\sqrt{1-\rho_{xy}(h)^2}}\,
        \exp\Big(-\frac{a_N^2}{1+\rho_{xy}(h)}\Big)\, \mathrm{d}h.
        }\]

     \rev{A differentiation gives
    }\[
    \rev{\frac{\mathrm d}{\mathrm dq}\log H_N(q)
    =\frac{q}{1-q^2}+\frac{a_N^2}{(1+q)^2}>0,
    }\]
     \rev{so $H_N$ is increasing. Moreover,
    $\big|\Sigma_{xy}(N,1)-\Sigma_{xy}(N,0)\big|=\rho_{xy}$.
    Since $\rho_{xy}(h)\le\rho_{xy}$, the preceding estimate implies
    }\begin{equation}\rev{\label{eq:T_Nbound}
    T_N\le C_S\sum_{x\ne y}\rho_{xy}H_N(\rho_{xy}).
    }\end{equation}
     \rev{For $k\ge1$, let
    }\[
    \rev{A_k:=\{(x,y)\in V_N^2:x\ne y,\ d_\T(x,y)=k\},
    \qquad \rho_k:=(r-1)^{-k}.
    }\]
    \rev{These are ordered pairs, consistently with the sum in \eqref{eq:T_Nbound}. }For each $x\in V_N$, the sphere of radius $k$ contains at most $r(r-1)^{k-1}$ vertices  \rev{and at most $N-1$ vertices. Hence
    }\[
    \rev{|A_k|\le N\min\{r(r-1)^{k-1},N-1\}.
    }\]
     \rev{Using $\rho_k\le\rho_1=1/(r-1)<1$ and the monotonicity of $H_N$, we obtain
    }\begin{align*}
    \rev{T_N
    }&\rev{\le C_SN H_N(\rho_1)
    \sum_{k\ge1}\min\{r(r-1)^{k-1},N-1\}(r-1)^{-k}}\\
    &\rev{\le C N\log N\,H_N(\rho_1).
    }\end{align*}
     \rev{Indeed, the sum is split at $\lfloor\log_{r-1}N\rfloor$: its first part has $O(\log N)$ bounded summands and its remaining part is a geometric tail. Finally,
    }\[
    \rev{H_N(\rho_1)
    =N^{-2/(1+\rho_1)}(\log N)^{O(1)},
    }\]
     \rev{because }$a_N^2=2\log N+O(\log\log N)$ \rev{. Therefore
    }\[
    \rev{T_N\le
    N^{\,1-\frac{2}{1+\rho_1}}(\log N)^{O(1)}
    =N^{(2-r)/r}(\log N)^{O(1)}=o(1),
    }\]
     \rev{for every }$r\ge3$ \rev{. }Theorem~\ref{thm:gauss-comp-ppp} proves the claim.
\end{proof}

\section{Extrema of the Gaussian \rev{free }field on random regular graphs}
\subsection{The zero-average Green function on \texorpdfstring{$\G_N$}{G N}}
Approximations for the Green function on random $r$-regular graphs are rather well-studied \rev{; }see e.g.\ \cite{Bauer19,AbacherliCerny2,Conchon23}. In this section, we prove that, for most vertices, the  Green function can be computed precisely. We  \rev{write $[N]:=\{1,\dots,N\}$, identify $V_N$ with $[N]$, and use $B_k(x;\G_N)$ }for the $k$-neighbourhood of  $x$ in $\G_N$. \rev{Here and below, with high probability means with probability tending to one as $N\to\infty$ under the graph law $\mathbf P$.
}\begin{prop}\label{prop:collection}
    There exist $k_1,k_2,k_3>0$, and $K_1<\infty$ such that the following properties are satisfied for the random graphs $\G_N$ with high probability under $\mathbf{P}$ as $N\to\infty$:
    \begin{enumerate}[(i)]
        \item $\G_N$ is a $k_1$-expander, and in particular connected.
        \item $ B_{k_1\log N}(x,\G_N)$ contains at most one cycle for all $x\in[N]$.
        \item The spectral gap of $\G_N$ is at least $k_2.$
        \item For all $x,y\in[N]$ \rev{,
        }\[
        \rev{|G_{\G_N}(x,y)|\le K_1\big((r-1)^{-d_{\G_N}(x,y)}\vee N^{-k_3}\big).
        }\]
    \end{enumerate}
\end{prop}
\begin{proof}
    Items (i) and (ii) are taken from \cite[Prop.\ 2.1]{Conchon23}. Item (iii) is discussed in \cite{CerTexWin11} \rev{. The one-sided upper bound in }(iv)  \rev{follows from \cite[Prop.\ 2.2]{AbacherliCerny2}. To control the negative part, split the defining integral at $t_N=(2/k_2)\log N$. Positivity of the heat kernel and the spectral-gap estimate \cite[Corollary~2.1.5]{Saloff97} give
    }\[
    \rev{G_{\G_N}(x,y)\ge-\frac{t_N}{N}
    -\int_{t_N}^{\infty}\left|\P_x^{\G_N}(\overline X_t=y)-\frac{1}{N}\right|\,\mathrm dt
    \ge-\frac{2\log N}{k_2N}-\frac{N^{-2}}{k_2}.
    }\]
    \rev{After decreasing $k_3<1$ and increasing $K_1$ if necessary, this lower bound combines with the cited upper bound to give the asserted absolute-value estimate}.
\end{proof}

We say that a vertex \( x \in [N] \) is \emph{\( \ell \)-good} if its \( \ell \)-neighbourhood \( B_\ell(x; \G_N) \) in \( \G_N \) is isomorphic (as a rooted graph) to the \( \ell \)-neighbourhood \( B_\ell(o; \T) \) of the root in the infinite \rev{rooted }$r$-regular tree. That is, \( x \) has no cycles within distance \( \ell \), and therefore the local structure around \( x \) is tree-like up to depth \( \ell \). We further call a vertex \( x \in [N] \) \emph{\( \ell \)-bad} if it is not \( \ell \)-good  \rev{and write
}\[
\rev{\mathcal B_N(\ell):=\left\{x\in[N]:x\text{ is }\ell\text{-bad}\right\}.
}\]
 \begin{lemma}\label{lem:badlem}
    Let  \rev{$1\le\ell=\ell(N)\le\frac{1}{2}\log_{r-1}N$}. There exists  \rev{a constant $K_2=K_2(r)<\infty$ such that, for every $z>0$,
    }\begin{equation}\rev{\label{eq:badbound}
        \mathbf{P}\big(|\mathcal B_N(\ell)|\geq z(r-1)^{2\ell}\big)\leq \frac{K_2}{z}.
    }\end{equation}
 \end{lemma}
\begin{proof}
    The proof is similar in vein to that of \rev{\cite[Lemma~4.2]{vdH1}; }here are the details\rev{. }The graph $\G_N$ has the same distribution as the \emph{configuration model} $\M_N$ derived from the degree sequence $(r,\dots,r)$ of length $N$\rev{, }conditioned on being loop-free and simple. \rev{To every $x\in[N]$ }attach $r$ formal objects $h_1(x),\dots,h_r(x)$ called \emph{half-edges}. \rev{Pick }a uniform matching on the set \rev{$\{h_i(x):1\le i\le r,\ 1\le x\le N\}$ }and identify a matched pair of half-edges with an edge. The resulting collection of edges can be identified with a multigraph $\M_N$ on the vertex set $[N]$. It is \rev{well known}, see e.g.\ \cite{vdH1}, that the probability that $\M_N$ is loop-free and simple is \rev{non-vanishing }as $N\to\infty$ \rev{.
}

    We explore the neighbourhood of vertex $1$ in  \rev{$\M_N$ }sequentially by revealing half-edges during a breadth-first search. More precisely, set  \rev{$\operatorname{exp}_0=\{1\}$ and let $\operatorname{act}_0$ contain the $r$ }half-edges  \rev{incident }to $1$ \rev{. If $t$ pairs have been revealed, choose the next active }half-edge  \rev{$H_{t+1}$ in breadth-first order and reveal its partner $H'_{t+1}$, which is uniform among the $rN-2t-1$ unmatched half-edges other than $H_{t+1}$. If }$H'_{t+1}$  \rev{belongs to an undiscovered vertex }$y_{t+1}$ \rev{, add $y_{t+1}$ to $\operatorname{exp}_t$, remove $H_{t+1}$, and add the remaining }unrevealed half-edges of  \rev{$y_{t+1}$ to obtain $\operatorname{act}_{t+1}$. If
    }\[
    \rev{H'_{t+1}\in\operatorname{act}_t\setminus\{H_{t+1}\},
    }\]
    \rev{a collision occurs at step $t+1$, and both matched half-edges are removed from the active set.
}

     \rev{In the absence of a collision, revealing all half-edges incident to vertices at distance at most }$\ell$  \rev{requires at most
    }\[
    \rev{m_\ell:=r\sum_{j=0}^{\ell}(r-1)^j
    =\frac{r((r-1)^{\ell+1}-1)}{r-2}
    }\]
     \rev{pairings. In particular, if $1$ is $\ell$-bad, then a collision occurs during these pairings. Conditional on no earlier collision, $|\operatorname{act}_t|\le r+(r-1)t$, and therefore
    }\[
    \rev{\mathbf P_{\mathrm{CM}}(\text{collision at step }t+1\mid\text{no earlier collision})
    \le\frac{r+(r-1)t}{rN-2t-1}.
    }\]
    \rev{The assumption on $\ell$ gives $m_\ell\le C_rN^{1/2}$ and hence
    }\[
    \rev{rN-2m_\ell-1\ge\frac{r}{2}N
    }\]
    \rev{for all sufficiently large $N$. Thus
    }\[
    \rev{\mathbf P_{\mathrm{CM}}(1\text{ is }\ell\text{-bad})
    \le\sum_{t=0}^{m_\ell-1}\frac{r+(r-1)t}{rN-2t-1}
    \le\frac{C_r(r-1)^{2\ell}}{N}.
    }\]
    \rev{Conditioning on simplicity only changes $C_r$. }By symmetry,
     \rev{$\mathbf E|\mathcal B_N(\ell)|\le C_r(r-1)^{2\ell}$, and \eqref{eq:badbound} follows from Markov}'s inequality.
\end{proof}

Next, we recall that the continuous time simple random walk $\overline{X}$ on a graph $G$ mixes exponentially fast if $G$ has a positive spectral gap.
\begin{lemma}[{\cite[Corollary\ 2.1.5.]{Saloff97}}]\label{lem:Saloff} Let $G$ be a connected $r$-regular graph of size $N$ with spectral gap $\kappa>0$ \rev{. Then, }for any vertices $x,y\in G$\rev{,
    }\[
    \Big|\P_x^G(\overline{X}_t=y)-\frac{1}{N}\Big|\leq \textup{e}^{-\kappa\, t} \quad\text{ for all }t>0.
    \]

\end{lemma}
Below, we will see that for the proof of Theorem~\ref{thm:graph-ext-proc} we only require precise control on the Green function of the $\ell_N$-good vertices which is our next goal.
\begin{lemma}\label{lem:goodgreenbound}
    Fix $\ell_0\in\N$ and let  \rev{$\widetilde\ell_N=\lfloor\varepsilon\log N\rfloor$ }with $\varepsilon>0$. There exists a constant $k_4=k_4(\ell_0,r,\varepsilon)>0$ such that
     \begin{equation}\rev{\label{eq:localapprox1}
    \mathbf P\left(
    \max_{\substack{x\in[N]:\,x\text{ is }\widetilde\ell_N\text{-good}\\
                      y\in[N]:\,d_{\G_N}(x,y)\le\ell_0}}
    \big|G_{\G_N}(x,y)-g(o,y')\big|
    \le N^{-k_4}\right)\longrightarrow1.
    }\end{equation}
     \rev{Here }$y'$  \rev{is any vertex of $\T$ at }depth $d_{\G_N}(x,y)$ \rev{, and a maximum over an empty index set is understood to be zero}.
\end{lemma}
\begin{proof}
    For $\ell\in\N$, let $g_\ell$ denote the Green function of continuous time simple random walk on $\T$ started at the root and killed upon reaching depth $\ell+1$. Let \rev{$L=\widetilde\ell_N$ and let }$\tau_N=\tau_N(x)$ denote the first time that  \rev{$\overline X$ visits $B_{L-1}(x;\G_N)^\mathsf c$ }when started in $x\in\G_N$. Because $x$ is assumed to be  \rev{$L$}-good, there is a perfect coupling between  \rev{$\overline X$ }under $\P_x^{\G_N}$ and under $\P_o^{\T}$ up until $\tau_N$ \rev{. Its occupation before $\tau_N$ is therefore $g_{L-1}(o,y')$. The strong Markov property gives
    }\[
    \rev{G_{\G_N}(x,y)
    =g_{L-1}(o,y')-\frac{\E_x^{\G_N}[\tau_N]}{N}
    +\E_x^{\G_N}\!\left[G_{\G_N}(\overline X_{\tau_N},y)\right].
    }\]
    The first \rev{error }term is of order $\log N/N$, since  \rev{$\overline X$ }has a linear drift upwards before $\tau_N$ \rev{. Moreover,
    $d_{\G_N}(\overline X_{\tau_N},y)\ge L-\ell_0$, so }Proposition~\ref{prop:collection} \rev{\textup{(iv)} yields
    }\[
    \rev{\left|\E_x^{\G_N}\!\left[G_{\G_N}(\overline X_{\tau_N},y)\right]\right|
    \le K_1\big((r-1)^{-(L-\ell_0)}\vee N^{-k_3}\big).
    }\]
     \rev{Finally, the strong Markov property on the tree and \eqref{eq:tree-green} show that
    }\[
    \rev{0\le g(o,y')-g_{L-1}(o,y')
    \le\sigma_r^2(r-1)^{-(L-\ell_0)}.
    }\]
    \rev{All bounds are uniform over the vertices in \eqref{eq:localapprox1}. Consequently,
    }\[
    \rev{\big|G_{\G_N}(x,y)-g(o,y')\big|
    \le C\left(\frac{L}{N}+(r-1)^{-(L-\ell_0)}+N^{-k_3}\right).
    }\]
    \rev{Choosing $0<k_4<\min\{1,\varepsilon\log(r-1),k_3\}$ }concludes the proof.
            \end{proof}

\subsection{Proof of main result}
    \begin{proof}[Proof of Theorem~\ref{thm:graph-ext-proc}]
    \rev{We may increase $K_1$ and hence assume $K_1\ge\sigma_r^2$. Put
    }\[
    \rev{\beta:=\frac{\sigma_r^2}{K_1},
    \qquad
    \varepsilon_0:=\frac{\sigma_r^2}{8K_1\log(r-1)},
    \qquad
    \ell_N:=\lfloor\varepsilon_0\log N\rfloor.
    }\]
    \rev{In particular, $\ell_N\le\frac{1}{2}\log_{r-1}N$ for all sufficiently large $N$. Choose a fixed integer $L_0$ so large that
    }\[
    \rev{\frac{2K_1}{\sigma_r^2}(r-1)^{-(L_0+1)}<\frac{1}{4}.
    }\]
    Define $\Omega'_N$ as  \rev{follows. We require }$\G_N$  \rev{to satisfy Proposition~\ref{prop:collection}}\textup{(i)  \rev{and (iv)}} \rev{and the high-probability event in }Lemma~\ref{lem:goodgreenbound}\rev{, }applied with $\varepsilon=\varepsilon_0$ and  \rev{$\ell_0=L_0$. Part~\textup{(i)} guarantees connectedness; parts~\textup{(ii)} and~\textup{(iii)} are not used below. Choosing $z=N^{\beta/4}$ }in Lemma~\ref{lem:badlem}  \rev{gives
    }\[
    \rev{\Omega''_N
    :=\left\{|\mathcal B_N(\ell_N)|
       \le N^{\beta/4}(r-1)^{2\ell_N}\right\}
    }\]
    \rev{occurs with $\mathbf P$-probability $1-o(1)$. }From now on, we work on a realisation
     \rev{$\G_N\in\Omega_N:=\Omega'_N\cap\Omega''_N$. In particular,
    }\[
    \rev{|\mathcal B_N(\ell_N)|\le N^{\beta/2}=o(N).
    }\]

    \rev{We first show that the bad vertices do not contribute to the extremes. Proposition~\ref{prop:collection}\textup{(iv)} implies
    $\operatorname{Var}^{\G_N}(\Psi_N(x))\le K_1$ for every $x$. For fixed $u\in\R$, let
    }\[
    \rev{u_N(u):=\sigma_r(a_N+b_Nu).
    }\]
    \rev{By }a Gaussian tail bound,
     \[
    \rev{\P^{\G_N}\big(\Psi_N(x)>u_N(u)\big)
    \le\exp\left(-\frac{u_N(u)^2}{2K_1}\right)
    =N^{-\beta+o(1)}.
    }\]
    \rev{Hence, }by a union bound,
     \begin{equation}\rev{\label{eq:bad-extremes}
    \P^{\G_N}\left(
    \max_{x\in\mathcal B_N(\ell_N)}\Psi_N(x)>u_N(u)\right)
    \le N^{-\beta/2+o(1)}=o(1).
    }\end{equation}
     \rev{This estimate holds for every fixed lower level $u$.
}

    It therefore suffices first to control the extremes over the $\ell_N$-good vertices, which we  \rev{denote by
    }\[
    \rev{V_N^{\mathrm{good}}:=}[\rev{N}]\rev{\setminus\mathcal B_N(\ell_N),
    \qquad m_N:=|V_N^{\mathrm{good}}|.
    }\]
    \rev{We have $m_N/N\to1$. Set, for $x\in V_N^{\mathrm{good}}$,
    }\[
    \rev{v_x:=G_{\G_N}(x,x),
    \qquad
    \widehat Z_x:=\frac{\Psi_N(x)}{\sqrt{v_x}}.
    }\]
     \rev{By Lemma~\ref{lem:goodgreenbound}, used with $y=x$,
    }\begin{equation}\rev{\label{eq:good-variance}
    v_x=\sigma_r^2+O(N^{-k_4})
    \qquad\text{uniformly in }x\in V_N^{\mathrm{good}}.
    }\end{equation}
    \rev{Let $(\eta_x)_{x\in V_N^{\mathrm{good}}}$ be }i.i.d.\  \rev{standard Gaussian variables}, independent of $\Psi_N$ \rev{, and define
    }\[
    \rev{P_{\widehat Z}
      :=\sum_{x\in V_N^{\mathrm{good}}}
        \delta_{(\widehat Z_x-a_N)/b_N},
    \qquad
    P_\eta
      :=\sum_{x\in V_N^{\mathrm{good}}}
        \delta_{(\eta_x-a_N)/b_N}.
    }\]
    \rev{Since $m_N/N\to1$, classical i.i.d.\ extreme-value theory with the unchanged norming sequences $a_N,b_N$ gives
    }\[
    \rev{\left(P_\eta,\max_{x\in V_N^{\mathrm{good}}}
    \frac{\eta_x-a_N}{b_N}\right)
    \Longrightarrow(\mathcal P_\infty,M_\infty).
    }\]

    \rev{For $x\ne y$ in $V_N^{\mathrm{good}}$, write
    }\[
    \rev{\rho_{xy}:=\E^{\G_N}[\widehat Z_x\widehat Z_y]
    =\frac{G_{\G_N}(x,y)}{\sqrt{v_xv_y}},
    \qquad q_{xy}:=|\rho_{xy}|.
    }\]
    Using Proposition~\ref{prop:collection}\textup{(iv)} and  \rev{\eqref{eq:good-variance}, there is }$C<\infty$ such that
     \begin{equation}\rev{\label{eq:good-correlation}
    q_{xy}\le C\big((r-1)^{-d_{\G_N}(x,y)}\vee N^{-k_3}\big).
    }\end{equation}
    \rev{We also need a bound strictly below one. If
    $1\le d_{\G_N}(x,y)\le L_0$, Lemma~\ref{lem:goodgreenbound} and \eqref{eq:tree-green} imply
    }\[
    \rev{q_{xy}\le(r-1)^{-d_{\G_N}(x,y)}+o(1)\le\frac{1}{2}+o(1).
    }\]
    \rev{If $d_{\G_N}(x,y)>L_0$, Proposition~\ref{prop:collection}\textup{(iv)}, the choice of $L_0$, and \eqref{eq:good-variance} give a bound smaller than $1/2$ for all sufficiently large $N$. Thus there is a constant
    }\begin{equation}\rev{\label{eq:c0}
    c_0\in(0,1)
    \quad\text{such that}\quad
    q_{xy}\le c_0
    }\end{equation}
    \rev{for all distinct $x,y\in V_N^{\mathrm{good}}$ and all sufficiently large $N$.
}

    Let $S\subset\R$ be a bounded open set. For $h\in[0,1]$,  \rev{set
    }\[
    \rev{X_x(h):=\sqrt{1-h}\,\eta_x+\sqrt h\,\widehat Z_x.
    }\]
    \rev{The pair $(X_x(h),X_y(h))$ }is a centred,  \rev{unit-variance }Gaussian vector with correlation  \rev{$h\rho_{xy}$. We argue }as in the proof of Theorem~\ref{thm:tree-ext-proc}  \rev{and aim to bound
    }\[
    T_N(\G_N)
    :=
    \frac{1}{b_N^2}\int_0^1
    \sum_{\substack{x,y\in V_N^{\mathrm{good}}\\x\ne y}}
    q_{xy}\,
    \P^{\G_N}\left(
      \frac{X_x(h)-a_N}{b_N}\in S,\
      \frac{X_y(h)-a_N}{b_N}\in S\right)\,\mathrm dh.
    \]
    Put $R_S=\sup_{s\in S}|s|$. Lemma~\ref{lem:doublegauss}, applied with the set $b_NS$, gives
    \[
    \rev{d(a_N,b_NS)\ge a_N-b_NR_S,
    \qquad
    (a_N-b_NR_S)^2=a_N^2-O_S(1).
    }\]
    \rev{Since
    }\[
    \rev{H_N(q):=\frac{1}{\sqrt{1-q^2}}
      \exp\left(-\frac{a_N^2}{1+q}\right)
    }\]
    is increasing on $[0,1)$, we obtain
     \begin{equation}\rev{\label{eq:graph-comparison-sum}
    T_N(\G_N)
    \le C_S\sum_{\substack{x,y\in V_N^{\mathrm{good}}\\x\ne y}}
    q_{xy}H_N(q_{xy}).
    }\end{equation}

                                                                                    As on the tree, we group  \rev{ordered }pairs by $k=d_{\G_N}(x,y)$ \rev{. For every $k\ge1$,
    }\[
    \rev{\left|\{(x,y)\in(V_N^{\mathrm{good}})^2:
    x\ne y,\ d_{\G_N}(x,y)=k\}\right|
    \le Nr(r-1)^{k-1}.
    }\]
     \rev{Set
    }\[
    \rev{k_\star:=\left\lfloor
      \frac{k_3\log N}{\log(r-1)}\right\rfloor.
    }\]
    \rev{For $k\le k_\star$, \eqref{eq:good-correlation} gives
    $q_{xy}\le C(r-1)^{-k}$. Consequently,
    }\[
    \rev{\varepsilon_N^{\mathrm{near}}
    \le CNk_\star H_N(c_0)
    =N^{\,1-\frac{2}{1+c_0}}(\log N)^{O(1)}
    =o(1).
    }\]
     \rev{Indeed, $c_0<1$ by \eqref{eq:c0}, whose proof uses Lemma~\ref{lem:goodgreenbound}, and therefore $2/(1+c_0)>1$}. For $k>k_\star$,  \rev{\eqref{eq:good-correlation} gives $q_{xy}\le CN^{-k_3}$, up to changing $C$ by a factor depending only on $r$, and hence
    }\[
    \rev{\varepsilon_N^{\mathrm{far}}
    \le CN^2N^{-k_3}H_N(CN^{-k_3})
    =O(N^{-k_3}\log N)=o(1).
    }\]
                                            Combining both bounds, we see that $T_N(\G_N)=o(1)$ \rev{. }Theorem~\ref{thm:gauss-comp-ppp} \rev{, applied to the $m_N$ coordinates and the levels $a_N,b_N$, now yields
    }\begin{equation}\rev{\label{eq:good-process-limit}
    \left(P_{\widehat Z},
    \max_{x\in V_N^{\mathrm{good}}}
    \frac{\widehat Z_x-a_N}{b_N}\right)
    \Longrightarrow(\mathcal P_\infty,M_\infty).
    }\end{equation}

    \rev{Finally, recall that $Z_x=\Psi_N(x)/\sigma_r$. By \eqref{eq:good-variance},
    }\[
    \rev{\frac{\sqrt{v_x}}{\sigma_r}=1+O(N^{-k_4})
    \quad\text{uniformly on }V_N^{\mathrm{good}}.
    }\]
    \rev{On every bounded extremal window, $\widehat Z_x=a_N+O(b_N)$, and
    }\[
    \rev{\left|
    \frac{Z_x-a_N}{b_N}
    -\frac{\widehat Z_x-a_N}{b_N}
    \right|
    =O(N^{-k_4}a_N^2)=o(1).
    }\]
    \rev{For a compact $K\subset\R$, every $Z$-atom in $K$ therefore corresponds to a $\widehat Z$-atom in an $o(1)$-enlargement of $K$; the number of such atoms is tight by \eqref{eq:good-process-limit}. For the maximum, the identity
    }\[
    \rev{\frac{(a_N+b_Nt)/(1+O(N^{-k_4}))-a_N}{b_N}
    =t+O(N^{-k_4}a_N^2)=t+o(1)
    }\]
    \rev{gives a two-sided sandwich. Hence \eqref{eq:good-process-limit} remains true with $\widehat Z_x$ replaced by $Z_x$.
}

    \rev{Finally, \eqref{eq:bad-extremes}, applied below any fixed compact set of levels, shows that the bad vertices contribute no extremal atom with probability tending to one. Combined with the tightness of the good maximum, it also restores the maximum over all vertices. This proves the quenched convergence on $\Omega_N$. Since $\mathbf P(\Omega_N^\mathsf c)=o(1)$, integration over the graph law gives the annealed convergence.
    }\end{proof}

\end{document}